\documentclass[letterpaper, 10 pt, conference]{ieeeconf}  

\IEEEoverridecommandlockouts                              
\usepackage{amssymb}
\usepackage{amsmath}
\usepackage{color}
\usepackage{graphics}
\usepackage{graphicx}
\usepackage{subfigure}
\usepackage{multirow}
\usepackage{algorithmic}
\usepackage{algorithm}

\def \kkk{\color{black}}
\def \R{\mathbb{R}}
\def \L{\mathcal{L}}
\def \E{\mathcal{E}}
\def \I{\mathcal{I}}

\def \S{\mathcal{S}}

\def \B{\mathcal{B}}
\def \I{\mathcal{I}}

\def \Z{\mathbb{Z}}

\def \V{\mathcal{V}}

\def \a{\alpha}
\def \b{\beta}
\def \lam{\lambda}

\def \s{\sigma}

\def \St{\S_{\tau}}

\def \1{\mathbf{1}}

\newtheorem{definition}{Definition}
\newtheorem{lem}{Lemma}
\newtheorem{cor}{Corollary}
\newtheorem{rem}{Remark}
\newtheorem{thm}{Theorem}

\title{\LARGE \bf
Domain of attraction of saturated switched systems under dwell-time switching
}

\author{Masood Dehghan
\thanks{This work was supported by the Agency for Science, Technology and Research (A*STAR) Funds  under Grant No. R-261-506-005-305).}
\thanks{M. Dehghan is with the department of Mechanical Engineering, National University of Singapore, 117576, Singapore. Email:
        {\tt\small (masood@nus.edu.sg)}}%
}

\begin{document}

\maketitle
\thispagestyle{empty}
\pagestyle{empty}

\begin{abstract}

This paper considers discrete-time switched systems  under dwell-time switching and in the presence of saturation nonlinearity.
Based on Multiple Lyapunov Functions  and using polytopic representation of nested saturation functions, a sufficient condition for asymptotic stability of such systems is derived. It is shown that this condition is equivalent to linear matrix inequalities (LMIs) and as a result, the estimation of domain of attraction is formulated into a convex  optimization problem with LMI constraints.
Through numerical examples, it is shown that the proposed  approach is less conservative than the others in terms of both minimal dwell-time needed for stability and the size of the obtained domain of attraction. \\

\end{abstract}

%


This paper considers the computation of  domain of attraction (DOA) of
discrete-time switched systems with saturation nonlinearity in the form of
\begin{align}
\label{eqn:switch-sat}
\left\{\begin{array}{l}
x(t+1) = A_{\s(t)} \, x(t) + B_{\s(t)} sat ( u(t) ) \\
\;\; \, \quad u(t) =  K_{\s(t)} \, x(t)
\end{array}\right.
\end{align}
where, $x \in \mathbb R^n$,  $u \in \mathbb R^m$  are the state and control variables respectively.
$\sigma(t) :  \Z^+ \rightarrow  \I_N:=\{1,\cdots, N\}$ is also a time-dependent switching
signal that indicates the current active mode of the system among $N$
possible modes in $\I_{N}$.
Symbol
$sat(\cdot)$ is the standard vector-valued saturation function, i.e., $sat(u) = [ sat(u_{1}) , \cdots, sat(u_{m}) ]^{T}$, with
$sat(u_{j}) = sgn(u_{j}) \min \{ 1, |u_{j}| \}$.
Without loss of generality, the   saturation limit  is normalized to one, by appropriately scaling the $B_{\s}$ and $K_{\s}$ matrices.

The study of switched systems has been quite active in the past decade due to their potential in modeling of
many practical real-life systems (see e.g.  car transmission systems \cite{Johansson-2004}, multiple-controller systems \cite{Morse1996}, genetic regulatory networks \cite{Jong-2004}, etc).
Most of the literature of the switched systems
is concerned with conditions that ensure stability of the
 system (\ref{eqn:switch-sat}) in the absence of saturation and when $\sigma(\cdot)$ is an arbitrary switching function \cite{switched_Lyapunov2002,Blanchini20072,Hua2010}.
Others consider stability of
switched systems when $\sigma(\cdot)$ satisfies some dwell-time restrictions \cite{Zhai2002,our11,Blanchini-2010,our33,Geromel20062,Chesi2012,our22}.

Since most of the physical actuators/sensors are subject to hardware limitations, presence of control saturation is always inherent to control systems, which may cause stability loss and performance degradation.
Moreover, computation and characterization of DOA  of such systems is specially challenging as their DOA is known to be
generally non-convex  \cite{book_saturation,Sun2007}.  Thus, estimation of DOA of switched systems in the presence of saturation nonlinearity  has received the attention of many researchers (see, e.g., \cite{Benzaouia2004,Benzaouia2006,Lu2008,Dehghan-sat}).

%


While several approaches have been proposed to handle saturation nonlinearity, two of them appear promising.
The first approach is to describe the saturation nonlinearity as a local sector bound nonlinearity with different multipliers
(see, e.g. \cite{Khalil2002,Tarbouriech2006}). Then, the S-procedure is used to derive sufficient conditions for stability of the resulted nonlinear system. The second approach,
is based on the polytopic representation of saturation nonlinearity \cite{Dasilva2001,Hu-Lin-saturation2002,Hu-Lin-discrete-saturation2002}, in which
 the  saturation function is represented as a linear differential/difference inclusion (LDI). With this representation, conventional tools designed  for linear systems can be used for saturated systems.
It has been realized that the second approach generally leads to less conservative results \cite{Zhou2011}.
Although the above mentioned approaches have been applied for switched systems under arbitrary switching
(see e.g. \cite{Benzaouia2004,Benzaouia2006,Lu2008,Jungers2011}),
  the extension of these methods for switched systems under dwell-time switching is not trivial due the complex structure of switching sequences that satisfy the dwell-time restriction.
  To the best of our knowledge there are very few results on such systems \cite{Ni-2010,Chen-2011}.

This paper  presents an LDI-based approach for  computation
of DOA of system (\ref{eqn:switch-sat}) when
$\sigma(\cdot)$ is a switching function that satisfies
the dwell-time restriction.
We formulate the problem into an optimization  with linear matrix inequalities (LMI) constraints that 
asymptotically stabilizes system (\ref{eqn:switch-sat}) and  at the same time enlarges its DOA. \kkk
We show that our result is less conservative than the others in terms of both minimal dwell-time needed for  stability  and the size of the obtained DOA.

In the limiting
case, where the dwell-time is one sample period, $\sigma(\cdot)$
becomes an arbitrary switching function, and our method retrieves the results of arbitrary switched systems appeared in the literature (see, e.g., \cite{Benzaouia2004,Benzaouia2006}).
Hence, this
work can also be seen as a generalization of those obtained for
arbitrary switched systems.

\kkk
The rest of this paper is organized as follows. This section ends
with a description of the notations used. Section \ref{sec:Preliminary}
reviews some standard terminology and preliminary results for switched
systems. Section \ref{sec:LDI} presents  the main results including the LMI formulation of the problem.
 Sections \ref{sec:examaple}  and  \ref{sec:conclusion} contain, respectively, numerical examples and conclusions.

The following notations are used. $\Z^+$ is the set of
non-negative integers. Given an integer $m \ge 1$, define
$\V_m:= \{S: S \subseteq  \{1,...,m\} \}$  as the set of all subsets of
$\{1,...,m\}$.
Clearly, $\{\emptyset\} \in \V_{m}$ and there are $2^{m}$ elements in the set $\V_{m}$.
Also let
 $S^c = \{j \in \{1,...,m\} : j \notin S\}$ to be the complement of $S$ in
$ \{1,...,m\}$.
Given $a>0$, the floor function $\lfloor a \rfloor$ is the largest integer that is less than $a$.
The $p$-norm
of a vector or a matrix is $\|\cdot\|_p, p=1,2,\infty$ with
$\|\cdot\|$ refers to the 2-norm and
  $\B_r := \{ x \in \R^n: \| x \| \le r \}$ is a norm ball with radius $r$.
Given a matrix $Y \in \R^{m \times n}$,
$Y^{i  \bullet }$ is the $i$-th row and  $Y^{\bullet j}$ is the $j$-th column of $Y$ and $\L(Y):= \{x : \| Yx \|_{\infty} \le 1 \}= \{ x : |Y^{i \bullet} x| \le 1 , \forall i=1 , \cdots , m  \}$.
The transpose of a vector/matrix $X$ is denoted by $X^T$ and
$I_m$ is the $m \times m$ identity matrix.
Positive definite (semi-definite) matrix, $P \in \R^{n
\times n}$, is indicated by $P \succ 0 (\succeq 0)$, $\E (P) := \{ x : x^{T} P x \le 1 \}$ and $\lambda_{\max} (P)$, $\lambda_{\min} (P)$ denote respectively the maximum and minimum eigenvalues of $P$.
Other notations are introduced when they are needed.

\section{Preliminaries}
\label{sec:Preliminary}

This section begins with the standard definitions of systems under dwell-time switching and  assumptions on the system, followed by
 preliminary stability results.


\begin{definition} \label{def:admissiblesequence}
Let a switching sequence of (\ref{eqn:switch-sat}) be denoted by  $\S(t)= \left\{ \s(t-1),\cdots,\s(1), \s(0) \right\}$
with switching instants at $t_{ 0},t_{ 1}, \cdots ,t_{ k}, \cdots$ with $t_{ 0} = 0$ and $t_{ k} < t_{ {k+1}}$.
 System (\ref{eqn:switch-sat}) has a dwell-time of $\tau$ if $t_{ {k+1}} - t_{ k} \ge \tau$ for all $k\in \Z^+$. In addition,
any switching sequence that satisfies this condition is
said to be dwell-time admissible (DT-admissible) with dwell-time $\tau$ and is denoted by $\St$.
\end{definition}

System (\ref{eqn:switch-sat}) is
assumed to satisfy the following assumptions:
\textbf{(A1)}  $A_i+B_{i}K_{i}$ is discrete-time Hurwitz for all $i \in \I_N$;
\textbf{(A2)} A value of $\tau \ge 1$ has been identified such that the unsaturated switched system (\ref{eqn:switch-sat}) is asymptotically stable with dwell-time $\tau$.

Assumption (A1) defines the family of systems considered in this work and is a reasonable requirement.
The presence of a minimal dwell-time that ensure asymptotic stability of system (\ref{eqn:switch-sat}) is well-known \cite{Zhai2002,our11}.
Hence, assumption (A2) is made out of convenience and poses no restriction.
In addition, it is assumed that there is no control on the  switching rule by the user, except that
the switching rule  satisfies the dwell-time consideration.
\kkk

In order to provide stability conditions for system (\ref{eqn:switch-sat}), 
additional notations  are required.
Consider the $i$-th mode of (\ref{eqn:switch-sat}).
Then the successor state of $x$, $F_{i}(x)$, under mode $i$ is
\begin{align}
\label{eqn:Fi}
F_{i}(x) = A_{i} \, x  + B_{i}  \, sat ( K_{i} x ).
\end{align}
Repeating the above leads to
\begin{align}
F_{i}^{2} (x) &= F_{i} ( F_{i} (x)) = A_{i} F_{i}(x)  + B_{i} \, sat ( K_{i} F_{i}(x) )  \nonumber \\
& \;\; \vdots \nonumber \\
 F_{i}^{t} (x) &=  F_{i} ( F_{i}^{t-1} (x)) = F_{i} ( F_{i} ( \cdots  F_{i} (x) ) )  \label{eqn:Fi-t}
\end{align}
where $F_{i}^{t}(x)$ is the state evolution of (\ref{eqn:switch-sat}) after $t$-steps with $x(0)=x$ and
$\St(t)= \{ i , i ,\cdots, i\}$.
Using this definition, the following result which is based on the Multiple Lyapunov Functions (MLFs) provides a sufficient condition for asymptotic stability of the origin of system (\ref{eqn:switch-sat}).

\begin{thm}
\label{thm:1}
Assume that, for some $\tau\ge1$, there exists a collection of positive definite matrices  $P_i \succ 0$ for each  $i \in \I_N$ such that
\begin{align}
&\big( F_i(x) \big) ^T P_i \, \big( F_i(x) \big) - x^T P_i \,x  < 0  \nonumber \\
& \qquad \qquad \qquad \qquad \qquad \qquad \qquad \quad  \forall x \neq 0, \forall i \in \I_N    \label{eqn:C1}\\
&\big( F_i^{\tau}(x) \big)^T P_j \, \big( F_i^\tau(x) \big) < x^T P_i \, x    \nonumber \\
& \qquad \qquad \qquad \qquad \forall x \neq 0, \forall (i,j) \in \I_N\times \I_N , i \neq j  \label{eqn:C2}
\end{align}
Then, the equilibrium solution $x = 0$ of saturated switched system (\ref{eqn:switch-sat})
 is globally asymptotically stable  with dwell-time $\tau$.
\end{thm}

\begin{proof}
Consider any DT-admissible switching sequence with dwell-time $\tau$ in accordance with Definition \ref{def:admissiblesequence}. Without loss of generality, assume that
$\s(t) = i$ for all $t \in [t_{ k} , t_{ {k+1}} )$ where $ t_{ {k+1}} =  t_{ {k}} +  \Delta_k$ and $\Delta_k \ge \tau$.  At $ t_{ {k+1}}$, system switches to mode $j$ and hence $\s( t_{ {k+1}} ) = j$.
Consider an associated Lyapunov function $V_i(x) = x^T P_i \, x$ for each mode $i \in \I_N$ and define
 $V(x(t)) := x(t)^T P_{\s(t)} x(t)$.
From (\ref{eqn:C1}), it follows that  $V(x(t+1)) - V(x(t))= V_{i}(x(t+1)) - V_i(x(t)) < 0$ is negative definite for all $t \in [t_{ k} , t_{ {k+1}} )$ along an arbitrary trajectory of (\ref{eqn:switch-sat})
and thus there exists a $\lam \in (0,1)$   and $\a >0$ such that
\begin{align}
\label{eqn:xV}
\| x(t) \|_2^2 \le \a \, \lam^{t- t_{ k}} \, V(x(t_{ k})), \quad \forall t \in [t_{ k} , t_{ {k+1}} )
\end{align}
On the other hand, from (\ref{eqn:C2}) it follows that
\begin{align}
V \big( x(t_{ {k+1}} ) \big) &=  \big( F_i^{\Delta_k}( x(t_k) )  \big)^T P_j \, \big( F_i^{\Delta_k}( x(t_k) )  \big) \nonumber  \\
& <  \big( F_i^{(\Delta_k - \tau)} (x  { ( t_k ) }   )  \big)^T  P_i \, \big( F_i^{(\Delta_k - \tau)} (x  { ( t_k ) }   )  \big) \nonumber \\
& <    x(t_k)^T P_i \, x(t_k)  =  V  \big( x(t_k) \big) \label{eqn:ML7}
\end{align}
where the second inequality follows from (\ref{eqn:C1}) and the fact that $\Delta_k - \tau \ge 0$.
Equation (\ref{eqn:ML7}) implies  that there exists a $\mu \in (0,1)$ such that $V(x(t_{ {k+1}} ) )  <  \mu  V ( x(t_k) )$ and thus
\begin{align}
V(x(t_{ {k+1}} ) )  <  \mu^k \,  V ( x(0) ) , \quad \forall k \in \Z^+
\end{align}
This together with (\ref{eqn:xV}) imply that the equilibrium solution $x=0$ of (\ref{eqn:switch-sat}) is asymptotically stable.
\end{proof}

While conditions (\ref{eqn:C1}) and (\ref{eqn:C2}) guarantee asymptotic stability of (\ref{eqn:switch-sat}), they are not tractable due to the existence of nested saturation functions in $F_i^\tau(x)$.
In the following section, the LDI representation of saturation function is explored to transform conditions of Theorem \ref{thm:1}  into linear matrix inequality (LMI) constraints that can be efficiently solved with convex optimization routines.

\section{Main Results}
\label{sec:LDI}

The LDI approach  is generalized in this section  and is used for  estimation of  DOA  of system (\ref{eqn:switch-sat}) under dwell-time switching.
LDI approach uses  auxiliary  terms and exploits their
convex hull to represent the saturation function as summarized in the following lemma:

\begin{lem}
\label{lem:sat-co}
\cite{Hu-Lin-discrete-saturation2002}
For any  $S \in \V_m$,
define $D_S$ to be the $m\times m$ diagonal matrix with
diagonal elements $D_S(j,j)$, whose value is 1 if $\: j \in S$ and
0 otherwise. Also define $D_{S^c} = {I}_m -D_S$. Then,
for all $u \in \mathbb R^m$ and $v \in \mathbb R^m$ such that
$\left|v_j\right| \le 1$ for all $j = 1, \cdots , m$:
\begin{align}
\label{eqn:sat-m}
sat(u) \in co \left\{ D_{S^c} u + D_{S} \, v :
\forall S \in  \V_m \right\}
\end{align}
\end{lem}

To illustrate the main idea of the LDI approach,  consider  any $u \in \R^2$ as an example.
According to Lemma \ref{lem:sat-co}, for any 
 $v = [ v_1 , v_2 ]^T \in \R^2$ 
such that  $|v_1| \le 1 ,  |v_2| \le 1$, it follows that
$$sat\left( \begin{bmatrix}
              u_1 \\
              u_2 \\
            \end{bmatrix}
 \right) \in co \left\{ \begin{bmatrix}
                          u_1 \\
                          u_2 \\
                        \end{bmatrix} , \begin{bmatrix}
                          u_1 \\
                          v_2 \\
                        \end{bmatrix} , \begin{bmatrix}
                          v_1 \\
                          u_2 \\
                        \end{bmatrix} , \begin{bmatrix}
                          v_1 \\
                          v_2 \\
                        \end{bmatrix}
 \right\}. $$
In other words, the above lemma states that $sat(u)$ can be expressed as a convex hull
of vectors formed by choosing some rows (those belonging to $S$) from $v$
and the rest (those belonging to $S^c$) from $u$.
Using (\ref{eqn:sat-m}) and assuming that $u = K_i x$ and $v$ is replaced by some linear function $H_i \,x$,
it follows that
\begin{align}
\label{eqn:FxKiLDI}
&sat(K_i \, x) \in co \left\{ D_{S^c} K_i \, x + D_{S} \, H_i \, x : \forall S \in  \V_m \right\}
\end{align}
for all $x \in \L(H_{i}) := \{x : | H_i^{j \bullet} x |  \le  1\} =
\{x : \| H_ix \|_\infty  \le  1\}$. 
Now, for a given $S \in \V_m$ define
\begin{align} \label{eqn:FxLDI}
E_{i, H_{i}} \big( x , S \big):= \Big( A_{i} + \sum_{j \in S^c}
B_{i}^{\bullet j} K_{i}^{j \bullet } \Big) x +  \Big( \sum_{j \in S} B_{i}^{\bullet j} H_{i}^{j \bullet }  \Big) x
\end{align}
and it follows from (\ref{eqn:Fi}), (\ref{eqn:FxKiLDI}) and  (\ref{eqn:FxLDI}) that for every $x \in \L(H_{i})$
\begin{align}
\label{eqn:Fixco}
&F_{i}(x) 
\in co \{ E_{i,H_{i}} (x,S) :  \forall S \in \V_{m} \}
\end{align}


While the LDI representation of $F_i(x)$ appeared in (\ref{eqn:C1}) is easily obtained from Lemma \ref{lem:sat-co},
the characterization of $F_i^\tau(x)$ appeared in  (\ref{eqn:C2}) is difficult as it consists of $\tau$ nested saturation functions.
The rest of this section describes the
characterization of $F_i^t(x)$ by introducing $t$ auxiliary variables $H_{i,1} , \cdots , H_{i,t}$.
Each of these variables
are introduced for LDI representation of one of the nested saturations.
%

Consider $F_i^2(x)$ and suppose that
$H_{i,1}$ and $H_{i,2}$ are associated for LDI representation of $sat(K_{i} x)$ and  $sat( K_{i} F_{i}(x) )$, respectively. Define
\begin{align} \label{eqn:GHxS}
&E_{i, H_{i,2}} \big( F_i(x) , S \big) := \nonumber \\
& \Big( A_{i} + \sum_{j \in S^c}
B_{i}^{\bullet j} K_{i}^{j \bullet } \Big) F_i(x) + \Big( \sum_{j \in S} B_{i}^{\bullet j} H_{i,2}^{j \bullet } \Big) x \quad
\end{align}
Then,  from (\ref{eqn:FxLDI})-(\ref{eqn:GHxS}), it follows that
\begin{align}
&F_{i}(x) \in co \{ E_{i,H_{i,1}} (x,S_{1}) : \forall S_{1} \in \V_{m} \},    \forall   x \in \L(H_{i,1})   \label{eqn:Fi-E1}\\
&F_{i}^{2} (x) = F_i (F_i(x)) \in co \{ E_{i, H_{i,2}} ( F_{i}(x),S_{2}) : \forall S_{2} \in \V_{m}\},  \nonumber \\
& \qquad \qquad \qquad \qquad \qquad \qquad \qquad \quad \;\; \forall x \in \L(H_{i,2}) \label{eqn:Fi-E2}
\end{align}
Since $F_{i}^{2} (x)$ is represented by the convex-hull of $E_{i, H_{i,2}} ( F_{i}(x),S_{2})$,
and $F_{i}(x)$ is by itself a convex-hull of $E_{i,H_{i,1}} (x,S_{1})$,  
it is straightforward (see Lemma \ref{lem:co_co} in the Appendix)
to expand  $F_{i}^{2} (x)$ as
\begin{align}
\label{eqn:Fi2H}
F_{i}^{2} (x)  \in co \{ E_{i, H_{i,2}} \big( E_{i ,H_{i,1}} (x, S_{1}) , S_{2} \big) : \forall S_{1} , S_{2} \in \V_{m} \}, \nonumber \\
\qquad  \forall x \in \L(H_{i,1}) \cap \L(H_{i,2}).
\end{align}

An example that illustrates this is given next.
Consider a single-input system  where $m = 1$ and hence $\V_{m} = \{ \{\emptyset\}, \{1\} \}$.
From (\ref{eqn:Fi2H}), it follows that
$ E_{i, H_{i,2}} \big( E_{i ,H_{i,1}}(x, S_{1}) , S_{2} \big)$ takes one of the following four expressions,  depending on the values of $S_1, S_2 \in \V_m$:


{\small
\begin{align}
&S_1 = \{\emptyset\},S_2 = \{\emptyset\} :  \; E_{i, H_{i,2}} \big( E_{i ,H_{i,1}} (x, \{\emptyset\}) , \{\emptyset\} \big) =  \nonumber \\
& \hspace{57mm} (A_{i} + B_{i} K_{i} )^{2}   x  \label{affine1} \\
&S_1 = \{\emptyset\},S_2 = \{1\}:  \; E_{i, H_{i,2}} \big( E_{i ,H_{i,1}} (x, \{\emptyset\}) , \{1\} \big) =  \nonumber \\
& \hspace{38 mm} A_i( A_i+B_iK_i)x   + B_i \, H_{i,2}  \, x \\
&S_1 =  \{1\},S_2 = \{\emptyset\}:  \; E_{i, H_{i,2}} \big( E_{i ,H_{i,1}} (x, \{1\}) , \{\emptyset\} \big) =  \nonumber \\
& \hspace{19.5 mm} ( A_i+B_iK_i)A_i \, x+ ( A_i+B_iK_i)B_i  \, H_{i,1} \,  x  \\
&S_1 = \{1\},S_2 = \{1\} :  \; E_{i, H_{i,2}} \big( E_{i ,H_{i,1}} (x, \{1\}) , \{1\} \big) =  \nonumber \\
& \hspace{34 mm} A_i^{2} \, x+ A_i B_i \, H_{i,1} \,x  + B_i \, H_{i,2} \, x \label{affine4}
\end{align}
}


Note that each one of the above expressions is an affine  function of $H_{i,1} x$, $H_{i,2}x$. Therefore, $F_i^2(x)$ which is the convex-hull of them, is also an affine  function of $H_{i,1}x$ and  $H_{i,2}x$. This  is a key property used for the conversion of condition (\ref{eqn:C2}) into an LMI (see Section \ref{sec:LMI-convert}).

Similar to the above procedure, by  associating auxiliary matrices $H_{i,1}$, $H_{i,2}$, $\cdots$, $H_{i,t}$ to  each one of the nested saturation functions appeared in $F_{i}^{t} (x)$, it follows that
\begin{align}
\label{eqn:Ft-Et}
F_{i}^{t} (x) \in co \Big\{
E_{i, H_{i,t}} ( \cdots (E_{i ,H_{i,1}} (x, S_{1}), \cdots) , S_{t} ): \quad \; \nonumber \\
  \forall S_{1} , \cdots , S_{t} \in \V_{m} \Big\},
 \; \forall x \in \L(H_{i,1}) \cap \cdots \cap \L(H_{i,t}).
\end{align}
To simplify the  notations of $F_i^t(x)$,  let
\begin{align}
&E_{i} \, (x,S_1) :=  E_{i ,H_{i,1}} (x, S_{1})  \nonumber  \\
&E_{i}^{2}(x,S_1,S_2) := 
E_{i, H_{i,2}} \big( E_{i ,H_{i,1}} (x, S_{1}) , S_{2} \big) \nonumber  \\
& \qquad \vdots  \nonumber \\
&E_{i}^{t}(x,S_1,\cdots,S_t) := 
 E_{i, H_{i,t}} ( \cdots (E_{i ,H_{i,1}} (x, S_{1}), \cdots) , S_{t} )
\label{eqn:Et}
\end{align}
With these notations,
the following theorem 
provides an estimate of  DOA of (\ref{eqn:switch-sat}).
\begin{thm}
\label{thm:2-sat}
Suppose for some $\tau \ge 1$, there exist a collection of $P_i \succ 0$ and matrices $H_{i,1}, H_{i,2}, \cdots, H_{i,\tau} \in \R^{m \times n}$ for each $i \in \I_{N}$ such that
\begin{align}
&\big[ E_{i} (x,S_1) \big]^{T} P_i \big[  E_{i} (x,S_1) \big]  -  x^{T} P_i \, x <  0  \qquad \;\; \nonumber \\
&\hspace{35 mm} \forall x \neq 0, \forall i \in \I_{N},  \forall S_1 \in \V_m   \label{eqn:MLF-E1} \\
&\big[ E_{i}^{\tau} (x,S_1,\cdots,S_\tau) \big]^{T} P_j \big[  E_{i}^{\tau} (x,S_1,\cdots,S_\tau) \big] -  x^{T} P_i \, x <  0     \nonumber \\
& \hspace{17 mm}   \forall x \neq 0, \forall i \neq j \in \I_N,  \forall S_1, \cdots, S_\tau \in \V_m  \label{eqn:MLF-E2} \\
&\E(P_i) \subseteq  \L (H_{i,t}) \hspace{15 mm} \forall i \in \I_{N}  ,  t =  1, 2, \cdots,  \tau
\label{eqn:LH}
\end{align}
Then,
(i) the origin of the saturated system (\ref{eqn:switch-sat}) with dwell-time $\tau$ is locally asymptotically  stable;
(ii) $\Psi := \bigcap_{i\in\I_N} \E(P_i)$ is inside the DOA of (\ref{eqn:switch-sat}).
\end{thm}

\begin{proof}
It is sufficient to show that for every $x \in \Psi$, equations (\ref{eqn:MLF-E1})-(\ref{eqn:LH}) imply (\ref{eqn:C1}) and (\ref{eqn:C2}).
To see this,
consider any arbitrary $x \in  \Psi = \cap_{i \in \I_N} \E(P_i)$. From
(\ref{eqn:LH}) it follows  that $x$ is inside the polyhedral region $ \L(H_{i,1}) \cap \cdots  \cap \L(H_{i,\tau})$ for all $i \in \I_N$.
This and (\ref{eqn:Fi-E1}), imply  that for every $x \in \Psi$,
$F_i(x) = \sum_{S_1 \in \V_m} \delta_{S_1} E_i(x,S_1)$,
 for some  $\delta_{S_1} \ge 0$ for each $S_1 \in \V_m$ such that $\sum_{S_1 \in \V_m} \delta_{S_1}=1$.
Since $E_i(x,S_1)^T P_i E_i(x,S_1)$ is a convex function,  we have
\begin{align*}
F_i(x)^T P_i F_i(x) &= \big[\sum_{S_1} \delta_{S_1} E_i(x,S_1) \big]^T P_i \big[ \sum_{S_1} {\delta_{S_1}} E_i(x,S_1) \big] \\
& \le \sum_{S_1} \delta_{S_1} \big[ E_i(x,S_1)^T P_i \, E_i(x,S_1) \big] \\
&< \sum_{S_1} \delta_{S_1} (x^T P_i \, x)   =  x^T P_i \, x
\end{align*}
where the last inequality follows from (\ref{eqn:MLF-E1}).

Similarly, from  (\ref{eqn:Ft-Et}) and (\ref{eqn:LH}) it is inferred that
$F_i^\tau(x) = \sum_{S_1,\cdots,S_\tau} \delta_{S_1,\cdots,S_\tau} E_i^\tau(x,S_1,\cdots,S_\tau)$,
for some $\delta_{S_1,\cdots,S_\tau} \ge 0$, $S_1 , \cdots, S_\tau \in \V_m$ such that $\sum_{S_1,\cdots,S_\tau} \delta_{S_1,\cdots,S_\tau}=1$.
Then from convexity of $E_i^\tau(x,S_1,\cdots,S_\tau)^T P_j E_i^\tau(x,S_1,\cdots,S_\tau)$ and (\ref{eqn:MLF-E2}), we have
${[F_i^\tau(x)]}^T P_j [F_i^\tau(x)]
\le \sum \delta_{S_1,\cdots,S_\tau} \big[ E_i^\tau(x,S_1,\cdots,S_\tau)^T P_j \, E_i^\tau(x,S_1,\cdots,S_\tau)^\tau \big]
< \sum \delta_{S_1,\cdots,S_\tau} (x^T P_i \, x  \big) =  x^T P_i \, x$.


Note that for every $x(0) \in \Psi$, $x(t)$ may move outside the $\Psi$ but condition (\ref{eqn:MLF-E2}) enforce that $x(t_1)$ (after the first switching) be inside  $ \mu \Psi$ for some $\mu \in (0,1)$.  In addition, condition (\ref{eqn:MLF-E1}) ensures that $x(t)$ remains inside the union of ellipsoids $\cup_{i \in \I_N} \E(P_i)$ for all $t$. This,  (\ref{eqn:MLF-E2}) and (\ref{eqn:LH}) together,  imply that $x(t)$ is inside polyhedral regions $\bigcap_{i \in \I_N} \big( \L(H_{i,1}) \cap \cdots, \L(H_{i,\tau}) \big)$ for all $t \in \Z^+$ and hence LDI representation of (\ref{eqn:Ft-Et}) is valid at all times.
\end{proof}


\begin{rem}
In the limiting case where  $\tau = 1$, $\s(\cdot)$ becomes an arbitrary switching function  and  the conditions of Theorem  \ref{thm:2-sat} retrieves the stability results appeared in the literature  for saturated  systems under arbitrary switching (see e.g.  \cite{Benzaouia2004,Lu2008}).
\end{rem}

\begin{rem}
Let $\bar A_i = A_i+B_i K_i$. Then, the conditions of Theorem \ref{thm:2-sat} in the absence of saturation become
\begin{align}
&\bar A_{i}^{T} P_i \,\bar A_{i}  - P_i  \prec  0,  \;  \forall  i   \label{eqn:LMI-no-sat1} \\
&\big[ \bar A_{i}^{\tau}  \big]^{T} P_j \big[  \bar A_{i}^{\tau}  \big]
- P_i  \prec  0,    \forall  i \neq j   \label{eqn:LMI-no-sat2}
\end{align}
which is the stability condition for (unsaturated) switched system appeared in \cite{Geromel20062}.
Thus, there indeed exist $P_i \succ 0$ and $H_{i,1}, \cdots, H_{i,2 \tau-1}$ that satisfy (\ref{eqn:MLF-E1})-(\ref{eqn:MLF-E2}) so long as LMI (\ref{eqn:LMI-no-sat1})-(\ref{eqn:LMI-no-sat2}) for system in the absence of saturation has a solution.
This also signifies assumption (A2).
\end{rem}

\subsection{LMI Formulation and Enlarging the Domain of Attraction}
\label{sec:LMI-convert}

The estimate of DOA of  system (\ref{eqn:switch-sat}) obtained from Theorem \ref{thm:2-sat} is the intersection of ellipsoidal sets $\E(P_i)$.
To enlarge the DOA, one must chose auxiliary matrices $H_{i,1}, \cdots, H_{i,\tau}$ and $P_i$'s such that the volume of
$\cap_{i\in \I_N} \E(P_i)$ is maximized. This can be done by solving the following constrained optimization problem:
\begin{align*}
\max_{P_i \succ 0,H_{i,1}, \cdots, H_{i,\tau}}  \text{volume }  \E(P_i) \nonumber \\
s.t. \; (\ref{eqn:MLF-E1}) ,  (\ref{eqn:MLF-E2}) \text{ and } (\ref{eqn:LH}).
\end{align*}

In the sequel, we describe how to transform the above optimization problem 
into
Linear Matrix Inequalities (LMIs) that can be  efficiently solved with LMI solvers (see e.g. \cite{cvx}).

The key point for this conversion is that $E_{i}^{t}(x,S_1,\cdots,S_t)$ for given $S_{1}, S_{2}, \cdots, S_{t} \in \V_{m}$,
is an affine function of variable $H_{i,1}x, \cdots, H_{i,t}x$. This means that $E_{i}^{t}(x,S_1,\cdots,S_t)$ can be rewritten as
\begin{align}
\label{Ei-affine}
&E_{i}^{t}(x,S_1,\cdots,S_t) = \Theta_{i,0}(S_1,\cdots,S_t)x +  \nonumber \\
& \hspace{5mm} \Theta_{i,1}(S_1,\cdots,S_t) H_{i,1} \,x
 +  \cdots +  \Theta_{i,t}(S_1,\cdots,S_t) H_{i,t} \, x
\end{align}
where $\Theta_{i, \cdot}(S_1,\cdots,S_t)$'s are only functions of $A_{i}, B_{i}, K_{i}$
(see e.g. (\ref{affine1})-(\ref{affine4}) for the expressions of $\Theta_{i,0}(S_1,S_2)$, $\Theta_{i,1}(S_1,S_2)$, $\Theta_{i,2}(S_1,S_2)$ for different values of $S_1$ and $S_2$).
Hereafter, the dependence of $\Theta_{i, t}$ on $(S_1, \cdots, S_t)$ is dropped for notational convenience unless needed.

Now, to transform (\ref{eqn:MLF-E2}) into an LMI constraint,
pre- and post-multiply it by $P_i^{-1}$. It follows that
\begin{align}
\label{eqn:LMI-interm}
&x^{T} \Big[ P_i^{-1} (\Theta_{i,0}  +  \cdots +  \Theta_{i,\tau} H_{i,\tau})^{T} P_j (\Theta_{i,0} +  \cdots +  \Theta_{i,\tau} H_{i,\tau})  P_i^{-1} \nonumber \\
& \hspace{10 mm} - P_i^{-1} \Big] x < 0 \qquad \forall x \neq 0 , \forall i \neq j
\end{align}
Let $Q_i = P_i^{-1}, Y_{i,1} = H_{i,1} P_i^{-1}, \cdots, Y_{i,\tau} = H_{i,\tau} P_i^{-1}$.
Then,  (\ref{eqn:LMI-interm})  is equivalent to
\begin{align*}
& \big( \Theta_{i,0} Q_i    +  \cdots +  \Theta_{i,\tau} Y_{i,\tau}  \big)^{T} Q_j^{-1}
\big( \Theta_{i,0} Q_i     + \cdots +  \Theta_{i,\tau} Y_{i,\tau} \big) \nonumber \\
& \hspace{55 mm} - Q_i \prec 0 \qquad  \forall i \neq j
\end{align*}
 Utilizing the Schur complement, this can be converted into
\begin{align}
\label{eqn:Lyap-LMI}
& \left[\begin{array}{cc} Q_i & * \\ \Theta_{i,0} Q_i +  \cdots +  \Theta_{i,\tau} Y_{i,\tau}  & Q_j \end{array}\right] \succ 0 & \forall i \neq j
\end{align}
where $*$ denotes the transpose of the off-diagonal term  and (\ref{eqn:Lyap-LMI}) is now an LMI in terms of the variables $Q_i, Q_j, Y_{i,1}, Y_{i,2}, \cdots, Y_{i,\tau}$.
Using the same procedure, constraint (\ref{eqn:MLF-E1}) is equivalent to
\begin{align}
\label{eqn:Lyap-LMI-1}
& \left[\begin{array}{cc} Q_i & * \\ \Theta_{i,0} Q_i +   \Theta_{i,1} Y_{i,1}  & Q_i \end{array}\right] \succ 0 & \forall i
\end{align}
Constraint (\ref{eqn:LH}) is also equivalent to the following  LMI constraints \cite{Hu-Lin-discrete-saturation2002}:
\begin{align}
\label{eqn:LH-LMI}
& \; \E(P_i) \subseteq \bigcap_{i \in \I_{N}, t \in \{ 1, \cdots, \tau\} }\L(H_{i,t}) \; \Leftrightarrow  \nonumber  \\
& \left[\begin{array}{cc} 1 & Y_{i,t}^{j \bullet } \\  * & Q\end{array}\right] \succeq 0  , \qquad \forall j \in \{ 1 , \cdots, m \} , \forall i \in \I_{N} , \nonumber  \\
& \qquad \qquad \qquad \qquad \quad \;\; \forall t \in \{ 1, \cdots, 2\tau-1\}
\end{align}
where $Y_{i,t}^{j \bullet }$ is the $j$-th row of $Y_{i,t}$.

Finally,  by using  $tr(P_i^{-1})$ as a measure of size of the ellipsoid $\E(P_i)$,
the following corollary provides an approach for enlarging the DOA of (\ref{eqn:switch-sat}).

\begin{cor}
\label{cor:3}
Suppose that for some  $\tau \ge1$, there exist  matrices $Q_i \succ 0$ and  $Y_{i,1}, \cdots, Y_{i,\tau}$ for $i = 1,2, \cdots, N$ such that the following linear matrix inequalities (LMIs) hold:
\begin{equation}
\label{eqn:LMI-DOA}
\left\{
\begin{array}{l}
\max_{Q_i \succ 0 ,Y_{i,1}, \cdots, Y_{i,\tau}} \; \sum_{i=1}^N  tr(Q_i)  \\
\\
\left[\begin{array}{cc} Q_i &  * \\ \Theta_{i,0}(S_1) Q_i +\Theta_{i,1}(S_1) Y_{i,1}  & Q_i \end{array}\right] \succ 0   \\
\hspace{60 mm}  \forall i , \forall S_1 \in \V_m \\
 \\
\left[\begin{array}{cc} Q_i & * \\ \Theta_{i,0}(S_1,\cdots, S_\tau) Q_i  +  \cdots +  \Theta_{i,t}(S_1,\cdots, S_\tau) Y_{i,\tau}  & Q_j \end{array}\right] \succ 0 \\
\hspace{40 mm}  \forall i \neq j , \forall S_1,\cdots, S_\tau \in \V_m\\
 \\
\left[\begin{array}{cc} 1 & Y_{i,t}^{j \bullet } \\  * & Q_i\end{array}\right] \succeq 0
\hspace{2mm} \forall i , \forall t \in \{1, \cdots, \tau\}, \forall j \in \{ 1 , \cdots, m \} \\
 \end{array}
\right.
\end{equation}
Then, the origin of switched system (\ref{eqn:switch-sat}) is locally asymptotically stable with dwell-time $\tau$ and  $\Psi = \bigcap_{i \in \I_N} \E(Q_i^{-1})$ is the estimate of DOA. The auxiliary matrices  $H_{i,t}$  are  obtained
from $H_{i,t} = Y_{i,t} P_i$ with $P_i = Q_i^{-1}$.

\end{cor}

\begin{rem}
In the optimization problem (\ref{eqn:LMI-DOA}),  $\sum_i tr(Q_i)$ is optimized
over all possible matrices $H_{i,t}, \cdots, H_{i,\tau}$, including $H_{i,t}=K_{i}$ for all $i\in \I_{N}$ and for all $t = 1, \cdots, \tau$. Hence, the resulting DOA is no smaller than the one tangential to the sides of the unsaturated region, i.e.
$\L_{K} := \cap_{i\in \I_{N}} \{ x : \|K_{i} x \|_{\infty} \le 1\}$.
\end{rem}

\begin{rem}
Any feasible solution of optimization problem (\ref{eqn:LMI-DOA}) with dwell-time $\tau$, is also a feasible solution for optimization problem (\ref{eqn:LMI-DOA})
with any $\bar \tau \ge \tau$.
This  means that $\Psi(\tau)$ is
a DOA of  (\ref{eqn:switch-sat}) with dwell-time $\bar \tau \ge \tau$ and
$\Psi(\tau) \subseteq \Psi(\bar \tau)$.
\end{rem}

\section{Numerical Example}
\label{sec:examaple}
The example considered is a single-input saturated switched system, taken from \cite{Chen-2011}, with $\I_N = \{ 1,2 \}$,
$A_{1} = \left[ \begin{smallmatrix} -0.7 \; & \;\;\, 1.0\\  -0.5 \; & -1.2 \end{smallmatrix} \right]$, $A_2=\left[ \begin{smallmatrix} 0.26\; & -1.0\\  1.7 \; & -1.5 \end{smallmatrix} \right]$,
$B_1=[1 ,\: 0]^T$, $B_2=[0 ,\: -1]^T$,
$K_1=[1.1759, \, 0.1089]$,
$K_2=[1.5114, \, -0.7765]$.

As LMIs (\ref{eqn:LMI-no-sat1})-(\ref{eqn:LMI-no-sat2}) admit a solution with $\tau = 2$, the  system is asymptotically stable with dwell-time $\tau = 2$ and thus assumption (A2) is satisfied for any $\tau \ge 2$. It can also be shown that
the  system is unstable under arbitrary switching and hence the methods proposed  for arbitrary switched systems are not applicable for this example.
The intention here is to compute an estimate of  DOA of the system from Corollary \ref{cor:3} for different values of dwell-time $\tau \ge 2$
 and compare them
  with the results presented in  \cite{Chen-2011}. 

The solution of the optimization problem (\ref{eqn:LMI-DOA})  with  $\tau =2$  are
$P_1 =  \left[ \begin{smallmatrix} 1.0839 & 1.5333  \\  *  & 3.1411 \end{smallmatrix} \right]$,
$P_2 =  \left[ \begin{smallmatrix} \;\;\;1.3408 & -0.7720  \\  *  & \;\;\;1.2585 \end{smallmatrix} \right]$,
$H_{1,1} =[0.8898  ,  0.7467]$, $H_{1,2} =[0.5660  ,  1.5560]$,
$H_{2,1} =[ 1.1270 ,  -0.8560]$, $H_{2,2} = -[0.3050  ,  0.4333]$.
Figure  \ref{fig:1-sat} shows the corresponding ellipsoidal sets $\E(P_1)$ and $\E(P_2)$ and the polyhedral regions $\L(H_{1,1}),  \L(H_{1,2}), \L(H_{2,1}), \L(H_{2,2})$.
Note that $\E(P_1) \subseteq \L(H_{1,1}) \cap \L(H_{1,2})$ and $\E(P_2) \subseteq \L(H_{2,1}) \cap \L(H_{2,2})$ as imposed by (\ref{eqn:LH-LMI}). The DOA together with a  sample trajectory of the system starting from $x(0)$ on the boundary of $\Psi = \E(P_1) \cap \E(P_2)$ under a periodic switching sequence is shown in Fig. \ref{fig:2-sat}.
Note that $x(t)$ may move out  of $\Psi$ (see $x(1), x(3) \notin \Psi$ in Fig.  \ref{fig:2-sat}) but $x(t)$ remains  in $\E(P_1) \cup \E(P_2)$ at all times.
The corresponding Lyapunov function $V(x(t))= x(t)^T P_{\s(t)} x(t)$ is also shown  in this figure.
Again, $V (t)$ is not monotonically decreasing with respect to $t$ but  $V(x(t_k))$ (the
points marked with ``o'') defines a   monotonically decreasing sequence and thus $V(t) \rightarrow 0$ as $t \rightarrow \infty$.

\begin{figure}[h]
\centering
  {\includegraphics[trim = 5mm 2mm 8mm 4mm, clip,width=.4\textwidth]{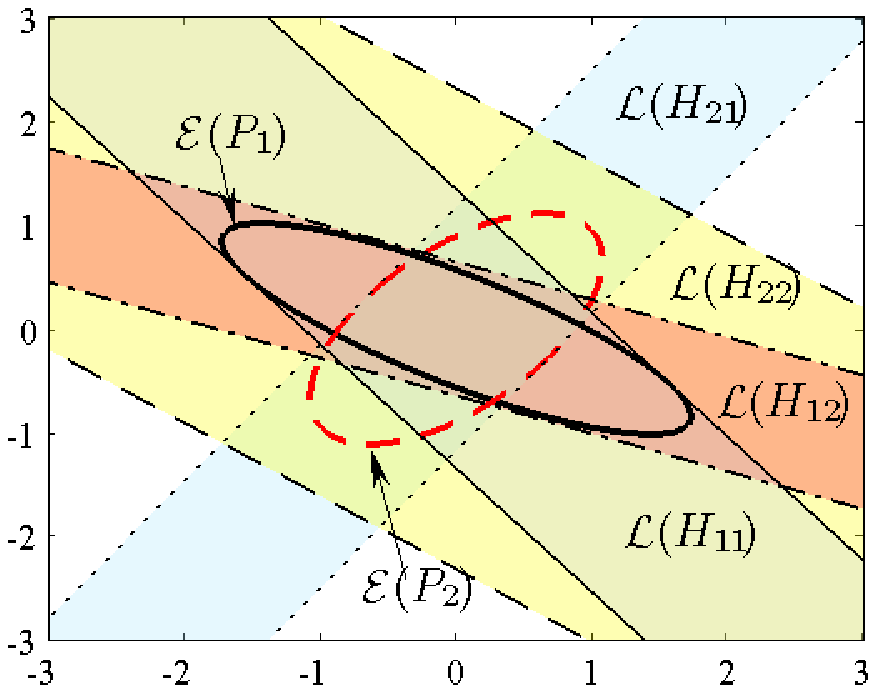}} \hfill
\caption{Illustration of $\Psi = \E(P_1) \cap \E(P_2)$  for $\tau =2$:   $\E(P_1) \subseteq \L(H_{1,1}) \cap \L(H_{1,2})$ and  $\E(P_2) \subseteq \L(H_{2,1}) \cap \L(H_{2,2})$;
}
\label{fig:1-sat}
\end{figure}

\begin{figure}[h]
\centering
  {\includegraphics[trim = 5mm 2mm 8mm 4mm, clip,width=.41\textwidth]{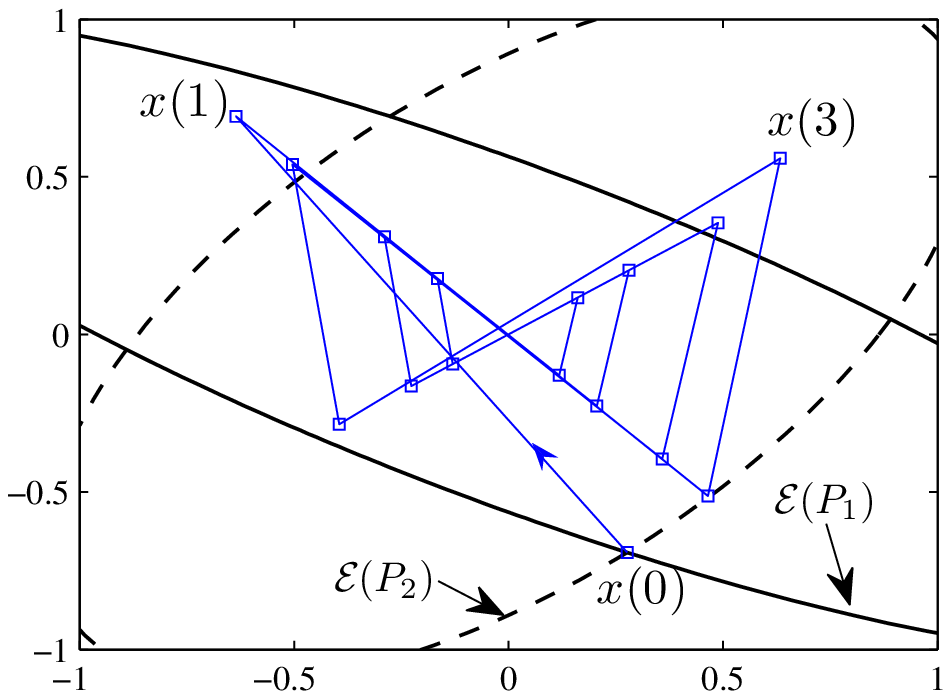}} \\
  {\includegraphics[trim = 7mm 40mm 8mm 2mm, clip,width=.4\textwidth]{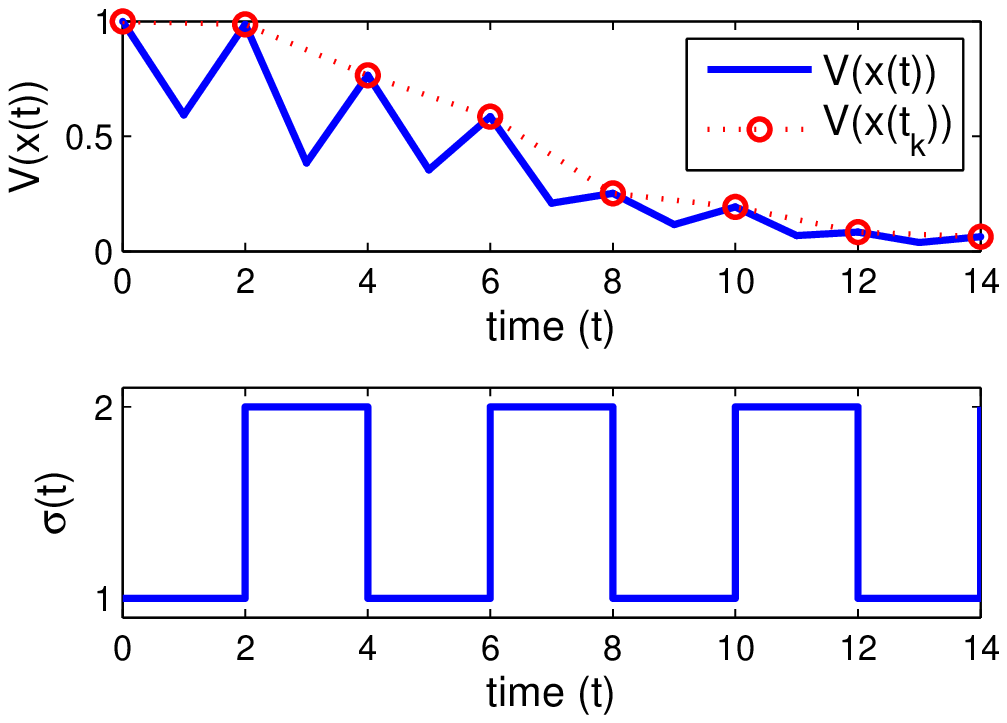}}
  \caption{(top) State trajectory from $x(0) = (0.2763 , -0.6918)$ on the boundary of $\Psi$ under a period switching with $\s(0)=2$, $t_{k+1}-t_k = 2, \forall k$, (bottom) The Lyapunov function $V(x(t))$ and the monotonically decreasing sequence $V(x(t_k))$ at switching times.}
\label{fig:2-sat}
\end{figure}

\subsection{Comparison with other methods}

As a comparison, the authors of \cite{Chen-2011}  use an LDI-based method  to obtain
an estimate of DOA of (\ref{eqn:switch-sat}).
They show that if there exist $\lambda \in (0,1)$, $\mu \ge 1$, $P_i \succ 0$ and $H_i$ for each $i \in \I_N$ such that
\begin{subequations}
\label{eqn:LLL}
\begin{align}
& [(E_{i,H_i} (x,S) )]^T P_i \,  [(E_{i,H_i} (x,S) )]  \le  \lambda \, x^\top P_i \, x  \nonumber \\
&\hspace{45 mm}  \forall i \in \I_N , \forall S \in \V_m \label{eqn:L1}  \\
& P_i  \preceq \mu \, P_j   \hspace{30mm} \forall (i,j) \in \I_N \times \I_N \label{eqn:L2} \\
&\E(P_i) \subseteq \L(H_i) \hspace{37mm}  \forall i \in \I_N \label{eqn:L3}
\end{align}
\end{subequations}
Then, equilibrium solution $x=0$ of (\ref{eqn:switch-sat}) is locally asymptotically  stable with dwell-time $\tau  \ge  \lfloor -\frac{\ln \mu }{\ln \lambda} \rfloor$.
For a fixed $\lambda$, conditions (\ref{eqn:L1}) and (\ref{eqn:L3})
can be easily  converted into LMIs using the same procedure developed in Section \ref{sec:LMI-convert} and
optimized such that the size of $\E(P_i)$'s are maximized.
 Then, an admissible choice of $\mu$ that satisfies (\ref{eqn:L2})  is
$\mu =  \max_{i,j} \frac{\lambda_{\max} (P_i)} {\lambda_{\min} (P_j)}$.  The estimate of DOA of this method is the largest norm-2 ball $\B_r = \{ x : \|x \| \le r \} \subseteq \cap_{i\in \I_N}\E(P_i)$
such that if $x(0) \in \B_r$ then $x(t) \in \cap_{i\in \I_N} \E(P_i)$ for all $t \in \Z^+$.
An admissible choice of $r$ that guarantees this condition is
$r = \min_{i \in \I_N}  \frac{1}{ \sqrt { \lambda_{max}(P_i) } }$.

For the example considered in this section, 
the smallest dwell-time $\tau $ that results in a feasible solution for the optimization problem (\ref{eqn:LLL}) is $\tau = 5$.  The resulting DOA, denoted by  $\B_r$, is shown in  Fig. \ref{fig:3-sat} and
  compared with the DOA obtained from Corollary \ref{cor:3} with $\tau = 5$.

\begin{figure}[h]
\centering
\includegraphics[trim = 8mm 4mm 10mm 6mm, clip,width=.4\textwidth]{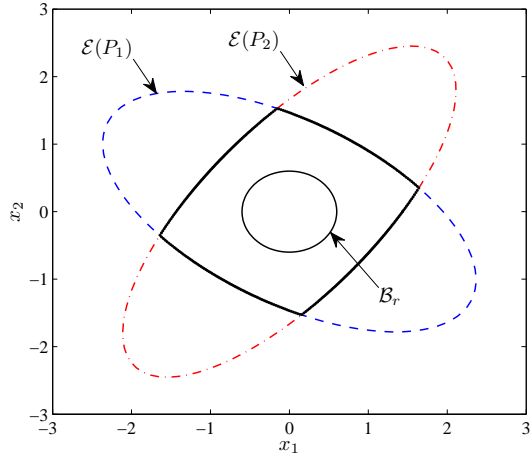}
\caption{Comparison of DOA  for $\tau =5$:  $\B_{r} \subset  \Psi = \E(P_1) \cap \E(P_2)$.}
\label{fig:3-sat}
\end{figure}

Computational results for  different values of $\tau $  are also presented in Table \ref{tab:1sat}. These results include the size of DOA and the total number of LMIs involved in each method.

\begin{table}[h]
\centering
\begin{tabular}{c|cc|cc}
\hline
              &\multicolumn{2}{c|}{Method of \cite{Chen-2011} }& \multicolumn{2}{c}{Corollary \ref{cor:3}} \\
  $\tau$ & $Area(\mathcal B_r)$  & \# LMI & $Area(\Psi)$ & \# LMI  \\
  \hline
  2 	  &  	-    	&   6 & 1.372  & 16   \\
  3 	  &  	-    	&   6 & 3.308  & 26   \\
  4 	  &  	-    	&   6 & 5.788 & 44   \\
  5 	  &  1.131	&   6 & 7.143  &  78  \\
  8 	  &  3.331  &  6  & 10.316  & 532   \\
\hline
\end{tabular}
\caption{Computational results}
\label{tab:1sat}
\end{table}

%
%
%
%
%

%
%
%
%

%
%

%
%
%
%
%


From Table \ref{tab:1sat}, it can be seen that the proposed LDI approach is less conservative,  in terms of both minimal dwell-time needed for stability and the size of DOA, than the LDI method of \cite{Chen-2011}.  This is mainly because the variables  $H_{i,1}, \cdots, H_{i,\tau}$ gives us more freedom to characterize the polytopic representation of the solution of system (\ref{eqn:switch-sat}) and  hence enable us to find a larger estimate of DOA. Of course, this is possible at the expense of a more computational effort as the number of LMI constraints involved in (\ref{eqn:LMI-DOA}) increases exponentially with $\tau$.


\section{Conclusion}
\label{sec:conclusion}

This paper proposes
a sufficient condition for asymptotic stability of discrete-time switched systems under dwell-time switching and in the presence of saturation nonlinearity.
This condition is shown to be
equivalent to linear matrix inequalities (LMIs). As a result,
the estimation of the domain of attraction is  formulated into an optimization problem
with LMI constraints.
Through numerical examples, it  is shown that our results are less conservative than the others, in terms of both minimal dwell-time needed for stability and the size of the obtained domain of attraction.

\addtolength{\textheight}{-14cm}


%

\appendix
\begin{lem}
\label{lem:co_co}
Let $\a \in co \{ \a_{i} : i =1,\cdots, n_\a \}$,
$\b \in co \{ \b_{j}: j =1, \cdots, n_{\b} \}$ and $\gamma = \a + \b$.
Then, $\gamma \in co \{ \a_{i}+\b_{j} : i \in \{1, \cdots, n_\a \} ,  j \in \{1, \cdots, n_\b \} \}$.
\end{lem}


\section*{Acknowledgment}

The financial support of A*STAR grant (Grant No. 092 156 0137) is gratefully acknowledged.

\bibliographystyle{IEEEtran}
\bibliography{referencessat}








\end{document}